\documentclass[10pt]{amsart}
\title{Constructing $G$-algebras}
\author{Kevin De Laet}
\address{Department of Mathematics, University of Antwerp \\ 
 Middelheimlaan 1, B-2020 Antwerp (Belgium) \\ {\tt kevin.delaet2@uantwerpen.be}}

\date{}
\usepackage{amsmath,amsfonts,array,amscd,amsthm,amssymb,stackrel,verbatim,wrapfig,subfig,float}
\usepackage{enumerate}
\usepackage{url}
\usepackage{tikz}
\usepackage[all]{xy}
\usetikzlibrary{arrows,matrix}
\usetikzlibrary{shapes.geometric}
\usetikzlibrary{decorations.markings} 

\tikzset{
  vertice/.style={circle,draw=black},
  decoration={markings,mark=at position 0.5 with {\arrow{>}}}
}

\newcommand{\wis}[1]{{\text{\em \usefont{OT1}{cmtt}{m}{n} #1}}}

\newcommand{\C}{\mathbb{C}}
\newcommand{\F}{\mathbb{F}}
\newcommand{\N}{\mathbb{N}}

\newcommand{\Z}{\mathbb{Z}}
\newcommand{\A}{\mathbb{A}}
\newcommand{\PP}{\mathbb{P}}

\newcommand{\V}{\mathbf{V}}
\newcommand{\T}{\mathbb{T}}
\newcommand{\mA}{\mathcal{A}}

\theoremstyle{plain}
\newtheorem{theorem}{Theorem}[subsection]
\newtheorem{lemma}[theorem]{Lemma}
\newtheorem{proposition}[theorem]{Proposition}
\newtheorem{corollary}[theorem]{Corollary}
\newtheorem{remark}[theorem]{Remark}
\newtheorem{example}[theorem]{Example}
\newtheorem{mydef}[theorem]{Definition}

\DeclareMathOperator{\Grass}{Grass}
\DeclareMathOperator{\Emb}{Emb}
\DeclareMathOperator{\Aut}{Aut}
\DeclareMathOperator{\Der}{Der}
\DeclareMathOperator{\Hom}{Hom}

\DeclareMathOperator{\im}{Im}

\DeclareMathOperator{\Ann}{Ann}

\numberwithin{equation}{section}
\begin{document}
\begin{abstract}
In this article we define $G$-algebras, that is, graded algebras on which a reductive group $G$ acts as gradation preserving automorphisms. Starting from a finite dimensional $G$-module $V$ and the polynomial ring $\C[V]$, it is shown how one constructs a sequence of projective varieties $\V_k$ such that each point of $\V_k$ corresponds to a graded algebra with the same decomposition up to degree $k$ as a $G$-module. After some general theory, we apply this to the case that $V$ is the $n+1$-dimensional permutation representation of $S_{n+1}$, the permutation group on $n+1$ letters.
\end{abstract}
\maketitle
\tableofcontents
\section{Introduction}
In noncommutative projective algebraic geometry, one studies graded algebras with good properties, most of the time of a homological nature. The most famous examples of such algebras are given by the Sklyanin algebras, that is, quadratic algebras of global dimension $n$ associated to a degree $n$ elliptic curve $E$ embedded in $\PP^{n-1}$ and a point $\tau \in E$. The Heisenberg group $H_n$ of order $n^3$ acts on these Sklyanin algebras, see for example \cite{odesskii1989sklyanin}, such that their degree 1 part is isomorphic to the Schr\"odinger representation $V$ of $H_n$ associated to some primitive $n$th root of unity $\omega$ and the relations are isomorphic to $V \wedge V$ as $H_n$-representation. These algebras also have the same Hilbert series as the polynomial ring in $n$ variables and are very well behaved.
\par Similarly, another famous class of noncommutative algebras with Hilbert series $\frac{1}{(1-t)^n}$ is given by quantum polynomial rings (see for example \cite{2015point}) defined by
$$
\C \langle x_1,\ldots,x_n \rangle /(x_i x_j -q_{ij} x_j x_i) , 1\leq i < j \leq n, q_{ij}\in \C^*.
$$
In this case, the $n$-dimensional torus group $\T_n = (\C^*)^n$ acts on these algebras with $V = \oplus_{i=1}^n \chi_{e_i}$ and again we have that the relations are isomorphic to $ V \wedge V$ as $\T_n$-representation.
\par These examples motivate the following questions:
\begin{itemize}
\item starting from a reductive group $G$ and $A$ a positively graded, connected algebra on which $G$ acts as gradation preserving automorphisms, can we construct new graded algebras $B$ such that $A \cong B$ as graded $G$-module? The most interesting case will be when $A = \C[V]$ with $V$ a finite dimensional representation of $G$.
\item Can we use the properties of $G$ or $V$ to say something about the constructed algebras? Are some isomorphic, how do point modules behave, $\ldots$
\end{itemize}
This paper shows some general theory regarding such algebras. One of the purposes of this paper is the study of subvarieties of grassmannians and maps between such varieties which are constructed this way, the key observation being that if $W \subset V $ is a subrepresentation of $V$ with $W \cong \oplus_{i=1}^k S_i^{\oplus e_i}$ and $V \cong \oplus_{i=1}^l S_i^{\oplus a_i}$ with obviously $k \leq l$, $e_i \leq a_i$  and $\dim \Hom_G(S_i,S_j) = \delta^i_j$, $\delta^i_j$ being the Kronecker delta, then
\begin{align*}
\Emb_G(W,V)&=\{\xymatrix{W \ar[r]^-f & V}| f \text{ is } G-\text{linear and injective}\}/\Aut_G(W)
\\ &\cong \prod_{i=1}^k \wis{M}^s_{a_i \times e_i}(\C)/\wis{GL}_{e_i}(\C)  \cong\prod_{i=1}^k \Grass(e_i,a_i)
\end{align*}
where $\wis{M}^s_{a\times b}(\C)$ are the $a \times b$-matrices of maximal rank. These objects are studied in Section 2. An example of maps between grassmannians that can occur will be studied in section 3 for quantum polynomial rings and 3-dimensional Sklyanin algebras.
\par In Section 4 we give some well known constructions of quadratic algebras with Hilbert series $\frac{1}{(1-t)^n}$ that can occur as $G$-deformations of the polynomial ring $\C[V]$ (the terminology will be explained in Section 2), like twisting with an automorphism or making Ore extensions. In the last section this theory is applied for $G = S_{n+1}$ the symmetric group of order $(n+1)!$ and its natural permutation representation coming from the action on $\{0,\ldots,n\}$. It will turn out that the `good' algebras will be either skew polynomial rings that will be twists of the commutative polynomial ring or differential polynomial rings.
\begin{remark}
Part of Section 2 was already proven in \cite{de2014character}, but the author feels that for completeness sake these results should be included in this paper.
\end{remark}
\subsection{Conventions and notations}
All the algebras we will study will be positively graded, connected and finitely generated in degree 1 $\C$-algebras, that is, 
$$
A = \C \oplus A_1 \oplus A_2 \oplus \ldots, A_i A_j = A_{i+j}.
$$
The last condition is equivalent to $A$ being generated over $\C$ by $V=A_1$ with $\dim V < \infty$. As such, we can write $A$ as
$$
A = T(V)/R \text{ with } T(V) = \C \oplus V \oplus (V \otimes V) \oplus \ldots, R \text{ homogeneous}
$$
By $\C[V]$ we denote the algebra $T(V)/(V \wedge V)$, the polynomial ring in $n$ variables with $\dim V = n$.
\par $G$ will always be a reductive group acting on $V$. In addition, we will assume that $G$ acts faithfully on $V$. If $g,h \in G$, then $[g,h] = ghg^{-1}h^{-1}$ is the commutator of $g$ and $h$.
\par If $x,y \in A$ with $A$ an algebra, then $[x,y] = xy-yx$ and $\{x,y\} = xy+yx$.
\par For $f_1,\ldots,f_k \in \C[V]$ for $V$ a vector space, $\V(f_1,\ldots,f_k)$ will be the Zariski closed subset of $\PP^{n-1}$ (if the elements are homogeneous) or $\A^n$ determined by $f_1,\ldots,f_k$, it will be clear from the context whether we are working in projective space or affine space.
\par $\T_k$ will be the $k$-dimensional torus $(\C^*)^k$. Let $\Z^k = \oplus_{i=1}^k \Z e_i$ and $\mathbf{a} \in \Z^k$, then $\chi_{\mathbf{a}}$ will be the character of $\T_k$ defined by 
$$
\chi_{\mathbf{a}}(t_1,\ldots,t_k) = t_1^{a_1}t_2^{a_2}\ldots t_k^{a_k}.
$$
\par For $\F$ a field, $\wis{GL}_n(\F)$, $\wis{PGL}_n(\F)$, $\wis{SL}_n(\F)$ or $\wis{PSL}_n(\F)$ will denote the general linear group, respectively the projective general linear group, special linear group or the projective special linear group. $\T_n$ will be the maximal torus in $\wis{GL}_n(\F)$ embedded as diagonal matrices and $\F^*$ will correspond to the scalar matrices. If $V$ is a finite dimensional $\F$-vector space, then we will also write $\wis{GL}(V)$, $\wis{PGL}(V)$, $\wis{SL}(V)$ or $\wis{PSL}(V)$.
\par $H_n$ will be the Heisenberg group of order $n^3$ for $n \geq 2$. This group is defined as
$$
H_n = \langle e_1,e_2|\langle e_1^n = e_2^n = [e_1,e_2]^n=1, [e_1,e_2] \text{ central}\rangle.
$$
$H_n$ fits in a short exact sequence
$$
\xymatrix{1 \ar[r]& \Z_n \ar[r] & H_n \ar[r] & \Z_n \times \Z_n \ar[r] & 1}.
$$
As such, the 1-dimensional representations of $H_n$ are labelled by $\chi_{a,b}$, with $\chi_{a,b}(e_1) = \omega^a$ and $\chi_{a,b}(e_2) = \omega^b$ for $\omega$ a fixed primitive $n$th root of unity. If $\omega$ is a primitive $n$th root of unity, then $H_n$ has a simple $n$-dimensional representation $V = \oplus_{i=0}^n \C x_i$ with the action defined by
$$
e_1 \cdot x_i = x_{i-1}, \qquad e_2 \cdot x_i = \omega^i x_i.
$$
A quick calculation shows that $[e_1,e_2]$ acts by multiplication with $\omega$. In particular, if $n = p$ prime, then all simple representations are described by these $p^2$ 1-dimensional characters and $(p-1)$ $p$-dimensional representations. For a complete description of simple representations of $H_n$ for $n\geq 2$ arbitrary, see \cite{grassberger2001note}.
\section{$G$-deformations}
\subsection{General theory}
\begin{mydef}
Let $G$ be a reductive group. We call a graded connected algebra $A$, finitely generated by degree 1 elements, a $G$-algebra if $G$ acts on it by gradation preserving automorphisms.
\end{mydef}
This implies that there exists a representation $V$ of $G$ such that $T(V)/R \cong A$ with $R$ a graded ideal of $T(V)$, which is itself a $G$-subrepresentation of $T(V)$. From now on, $A$ is a $G$-algebra.
\par In this setting, as $G$ is reductive, each graded component $A_k$ has a decomposition in simple $G$-representations.
\begin{mydef}
Let $A$ be a $G$-algebra. Then $B$ is a $G$-deformation of $A$ up to degree $k$ if $B$ is a $G$-algebra and we have 
$$
\forall 0 \leq i \leq k: A_i \cong B_i \text{ as $G$-representations.}
$$
If $\forall k \in \N: A_k \cong B_k$, then we call $B$ a $G$-deformation of $A$. 
\end{mydef}
We will always assume that, if $A_1 \cong V$, then $B_1 \cong V$ as $G$-representations. Equivalently, the relations we will deform are always of degree larger than or equal to 2.
\par If $A$ is a $G$-algebra, then $A$ determines for each $k \geq 2$ a short exact sequence of $G$-morphisms
$$
\xymatrix{0 \ar[r]& \ker(\phi_k) \ar[r] & T(V)_k \ar[r]^-{\phi_k} & A_k \ar[r] & 0}
$$
Let $A_k \cong \oplus_{i_k=1}^{n_k} S_{i_k}^{ e_{i_k}}$ and $T(V)_k = \oplus_{i_k=1}^{n_k} S_{i_k}^{ a_{i_k}}$ as $G$-representations with $e_{i_k} \leq a_{i_k}$, with some $e_{i_k}$ possibly equal to 0.
Then $\ker(\phi_k) \cong \oplus_{i_k=1}^{n_k} S_{i_k}^{ f_{i_k}}$ with $f_{i_k} = a_{i_k}-e_{i_k}$. Therefore, if $B$ is a $G$-deformation of $A$ up to degree $k$, then $B$ determines a $G$-subrepresentation of $T(V)_j$ for each $2\leq j  \leq k$, isomorphic to $\ker(\phi_j)$. Such subrepresentations are determined by
$$
\Emb_G(\ker(\phi_j),V^{\otimes j}) = \left\{ \xymatrix{\ker(\phi_j)	 \ar[r]^-f & V^{\otimes j}}| f \text{ injective, $G$-linear} \right\}/\Aut_G(\ker(\phi_j)).
$$
Due to Schur's lemma, using the decompositions from above, this set is the same as
$$
\prod_{i_j = 1}^{n_j} \wis{M}^s_{a_{i_j}\times f_{i_j}}(\C)/\wis{GL}_{f_{i_j}}(\C) \cong \prod_{i_j = 1}^{n_j} \Grass(f_{i_j},a_{i_j}).
$$
\par We see that a $G$-deformation $B$ up to degree $k$ of an algebra $A$ determines a unique point in 
$$
\prod_{j=2}^k \prod_{i_j = 1}^{n_j} \Grass(f_{i_j},a_{i_j}).
$$
\begin{theorem}
The $G$-deformations up to degree $k$ of $A$ are parametrized by a projective variety.
\end{theorem}
\begin{proof}
The proof is by induction. For any degree $k$, put
$$
T(V)_k \cong \oplus_{i_k=1}^{n_k} S_{i_k}^{ a_{i_k}}, A_k\cong \oplus_{i_k=1}^{n_k} S_{i_k}^{ e_{i_k}}
$$
and we allow $e_{i_k}=0$. Let $f_{i_k} = a_{i_k}-e_{i_k}$ be the multiplicities of the simple representations in $\ker(\phi_k)$ with as above $\phi_k$ being the natural projection map
$$
\xymatrix{ T(V)_k \ar[r]^-{\phi_k}& A_k}.
$$
Let $\mathbf{V}_k$ be the set parametrizing deformations up to degree $k$ of $A$.
\par First consider $k=2$. Then from the above discussion it follows that the deformations up to degree 2 are parametrized by $\prod_{i_2 = 1}^{n_2} \Grass(f_{i_2},a_{i_2})$, which is clearly a projective variety.
\par Assume now that $\mathbf{V}_{k-1}$ is a projective subvariety of 
$$
\prod_{j=2}^{k-1} \prod_{i_j=1}^{n_j} \Grass(f_{i_j},a_{i_j}).
$$
Then a point $(P_2,\ldots,P_{k-1},P_k) \in \prod_{j=2}^{k} \prod_{i_j=1}^{n_j} \Grass(f_{i_j},a_{i_j})$ with $P_i \subset T(V)_i$ is an element of $\mathbf{V}_{k}$ if and only if $(P_1,\ldots,P_{k-1}) \in \mathbf{V}_{k-1}$ and for each $2\leq i \leq k-1$ and each $0\leq l \leq k-i$, $V^{\otimes l} \otimes P_i \otimes V^{\otimes k-i-l} \subset P_k$. This is clearly a closed condition, so $\mathbf{V}_{k-1}$ is indeed closed.
\end{proof}
As in the theorem, let $\mathbf{V}_k$ be the variety parametrizing $G$-deformations of $A$ up to degree $k$.
\par We have natural morphisms coming from projection maps between grassmannians
$$
\xymatrix{\ldots \ar[r]^-{\psi_{k+1}} &\V_k \ar[r]^-{\psi_k}& \V_{k-1}\ar[r]^-{\psi_{k-1}}& \ldots}
$$
such that $\V = \varprojlim \V_k$ parametrizes $G$-deformations of $A$.
\begin{mydef}
We say that a connected variety $Z$ parametrizes $G$-deformations up to degree $k$ of a $G$-algebra $A$ if $Z$ can be embedded in $V_k$ for some $k \geq 2$ by some morphism $\xymatrix{Z \ar[r]^-\alpha & V_k}$ such that
the point corresponding to the algebra $A$ is in $\im(\alpha)$. If for each point $z \in Z$, the algebra $A^z=T(V)/(\alpha(z))$ has the property that
$$
\forall k \in \N: (A^z)_k \cong A_k \text{ as $G$-representations},
$$
then we say that $Z$ parametrizes $G$-deformations of $A$.
\end{mydef}
\begin{proposition}
Let $C$ be a smooth projective curve parametrizing $G$-deformations up to degree $k$ such that there exists an open subset $U \subset C$ parametrizing $G$-deformations of $A$, let $\xymatrix{C\ar@{^{(}->}[r]^-{\alpha_k} & \V_k}$ be the corresponding embedding. Then $C$ naturally parametrizes $G$-deformations of $A$, by which we mean that there exists a natural embedding $\xymatrix{C \ar@{^{(}->}[r]& \V}$ extending $\alpha_k$.
\end{proposition}
\begin{proof}
We have the following commutative diagram
$$
\xymatrix{\ldots \ar[r]^-{\psi_{k+2}} & \V_{k+1}\ar[r]^-{\psi_{k+1}} & \V_k \ar[r]^-{\psi_{k}}& \ldots\\
 &  & C \ar@{^{(}->}[u]^-{\alpha_k} \ar@{-->}[lu]^-{\alpha_{k+1}} &}
$$
with $\alpha_k$ an embedding and $\alpha_{k+1}$ a rational morphism. Due to \cite[Proposition 2.1]{ArithmeticElliptic2009}, $\alpha_{k+1}$ can be extended to a morphism of $C$ into $\V_{k+1}$ which we also denote as $\alpha_{k+1}$. As $\psi_{k+1}\circ \alpha_{k+1}$ coincides with $\alpha_k$ on an open (and hence dense) subset of $C$, they coincide on $C$. As $\alpha_{k}$ is an injection, $\alpha_{k+1}$ is also an injection.
\par By induction, we get commuting triangles $\forall k \geq 2$ with $\alpha_k$ an embedding for all $k$ large enough, so we indeed have an embedding of $C$ into $\V$.
\end{proof}
\begin{example}
Let $A = \C[x,y,z]/(\{x,y\}\{y,z\},\{z,x\})$, put $V = \C x \oplus \C y \oplus \C z$ and consider the group $G$ generated by $S_3\subset \wis{GL}_3(\C)$ as permutation matrices and the order 3 element
$$\begin{bmatrix}
1&0&0 \\0 & \omega & 0 \\ 0 & 0 & \omega^2
\end{bmatrix}.
$$
Then $G \cong H_3 \rtimes \Z_2$, with the action of $\Z_2= \langle t \rangle$ defined by $t \cdot e_1 = e_1^{-1}$, $t \cdot e_2 = e_2^{-1}$. Then $V \otimes V$ decomposes as $G$-representation as $W^{\oplus 2} \oplus P$ and $V \wedge V \cong P$, so the $G$-deformations of $A$ are parametrized by $\PP^1$. These algebras parametrized by $\V_2 \cong \PP^1$ have relations
$$
\begin{cases}
A(yz+zy)+Bx^2,\\
A(zx+xz)+By^2,\\
A(xy+yx)+Bz^2.
\end{cases}
$$ 
As in \cite{de2015graded}, generically these are Artin-Schelter regular graded Clifford algebras. There are 4 points where the corresponding algebra doesn't have the correct Hilbert series:
$$
S=\{[0:1],[1:1],[1:\omega],[1:\omega^2]\} \text{ with } \omega^3=1, \omega \neq 1.
$$
For example, in $[0:1]$, we have that $[\{x,y\},z]$ and cyclic permutations of this relation are not implied by the relations $x^2=y^2=z^2$, although they are implied by all other relations on $\PP^1$. So a natural extension of the rational map
$$
\xymatrix{\PP^1 \ar@{^{(}->}[r]^{\alpha_3}& \V_3}
$$
to the point $[0:1]$ is by adding the cyclic permutations of $[\{x,y\},z]$ as extra relations.
\end{example}
A trivial consequence of a connected variety $Z$ parametrizing $G$-deformations is that for each $i \in \N$ the function 
\begin{gather*}
\xymatrix{Z \ar[r]^-{\beta_i}& \N},
\xymatrix{z \ar@{|->}[r] &\dim_\C A^z_i}
\end{gather*}
is constant. A natural question to ask is the following: if $Z$ parametrizes $G$-deformations up to degree $k$ for some $k \in \N, k \geq 2$ and $\forall i \in \N: \beta_i(z) = \dim_\C A_i$, does $Z$ parametrize $G$-deformations of $A$? 
\begin{lemma}
Let $A$ be a $G$-algebra with $\xymatrix{ T(V) \ar@{->>}[r]^-{p} & A}$ the natural projection map. Let $A_k = \oplus_{i_k=1}^{n_k} S_{i_k}^{e_{i_k}}$ be the decomposition of $A_k$ into simple $G$-representations and similarly $T(V)_k = \oplus_{i_k=1}^{n_k} S_{i_k}^{a_{i_k}}$ with naturally $a_{i_k} \geq e_{i_k}$. Then there exists a subspace $W \subset T(V)_k$ such that $W \cong A_k$ as $G$-representations and $p|_W$ is an isomorphism of $G$-representations.
\end{lemma}
\begin{proof}
It follows from Schur's lemma that the map $\xymatrix{ T(V)_k \ar@{->>}[r]^-{p_k} & A_k}$ is a surjective element of $\oplus_{i_k=1}^{n_k}\Hom_G(S_{i_k}^{a_{i_k}},S_{i_k}^{ e_{i_k}})\cong\oplus_{i_k=1}^{n_k} \Hom(\C^{a_{i_k}},\C^{e_{i_k}})$. There it reduces to a statement of linear maps, which follows from standard linear algebra.
\end{proof}
\begin{theorem}
Let $Z$ be an irreducible variety parametrizing $G$-deformations up to degree $k$ of $A = T(V)/(R)$. For a point $z \in Z$, let $A^z = T(V)/(\alpha(z))$ be the corresponding associative algebra with $R^z \in \mathbf{V}_k$. If $\forall z \in Z: H_{A^z}(t) = H_{A}(t)$, then we have
$$
\forall i \in \N: (A^z)_i \cong A_i \text{ as $G$-representations}.
$$
\end{theorem}
\begin{proof}
For $x \in Z$, let $\xymatrix{ T(V) \ar@{->>}[r]^-{p^x} & A^x}$ denote the natural homomorphism. Assume that $(A^x)_l \not \cong (A^y)_l$ for some $l > k$. Then we can choose $l$ minimal with this property. According to the previous lemma we can find a subspace $W$ of $T(V)_l$ such that $p^x(W) = (A^x)_l$ and $W \cong (A^x)_l$ as $G$-representations.
\par Using the fact that $\forall z \in Z: H_{A^z}(t) = H_{A}(t)$, there exists a Zariski open subset $U$ with $x \in U$ of $Z$ such that
$$
\forall z \in U: p^z(W) = (A^z)_l.
$$
As $p^z$ is a $G$-morphism, we have $W \cong (A^z)_l$.
\par Similarly there exists an open subset $U'$ with $y \in U'$ and a subspace $W'$ of $T(V)_l$ such that $W' \cong (A^y)_l$ and $((p^z)_l)|_{W'}$ a $G$-isomorphism $\forall z \in U'$. As $Z$ was irreducible, there exists a point $a \in U \cap U'$. But then it follows that
$$
(A^x)_l \cong W \cong (A^a)_l \cong W' \cong (A^y)_l
$$
which is a contradiction.
\end{proof}
A trivial consequence of this theorem follows.
\begin{corollary}
Let $Z$ be a connected variety that parametrizes graded algebras $A^z, z \in Z$ with constant Hilbert series. Then for every 2 points $x,y \in Z$ and every $k \in \N$, we have isomorphisms $(A^x)_k \cong (A^y)_k$ as $G$-representations.
\label{cor:constchar}
\end{corollary}
\subsection{Symmetries on $\V_k$}
Assume that $A = T(V)/R$ is a $G$-algebra and let $H = N_{\wis{GL}(V)}(G) = \{h \in \wis{GL}(V)|hgh^{-1} \in G \}$. Let $H'$ be the maximal subgroup of $H$ such that $A$ is also a $H'$-algebra, so in particular $G \subset H'$. We have a morphism $\xymatrix{H' \ar[r] & \Aut(G)}$ defined by $h\mapsto \varphi_h,\varphi_h(g) = hgh^{-1}$. This implies that we can twist every representation $\xymatrix{G \ar[r]^f& \wis{GL}(W)}$ with $\varphi_h$
$$
\xymatrix{G \ar[r]^-{\varphi_h} & G \ar[r]^-f & \wis{GL}(W)},
$$
which defines a new representation $f\circ \varphi_h$ of $G$, possibly non-isomorphic to $W$ (except if $V = W$ of course). Denote this new $G$-representation by $W^h$. It is obvious that if $W$ is simple, then $W^h$ is also simple. However, if $A$ is an $H'$-algebra, this implies that if $S^{e}$ with $S$ simple is a $G$-subrepresentation of $R_k$ for some $k \geq 2$ with $e$ the multiplicity of $S$ in $R_k$, then $(S^h)^e$ for each $h \in H$ is also a subrepresentation of $R_k$ and $e$ is the multiplicity of $S^h$ in $R_k$.
\begin{proposition}
Let $A=T(V)/R$ be a $G$-algebra and $H,H'$ as before. Then $H'$ acts on $\V_k$, the variety parametrizing $G$-deformations up to degree $k$ such that points in the same orbit correspond to isomorphic algebras.
\label{prop:symVk}
\end{proposition}
\begin{proof}
Let $h \in H'$, $B$ a $G$-deformation of $A$. Define a new action of $G$ on $B$ by $g \cdot v = hgh^{-1}v$, then $B$ is a $G$-algebra under this action, with the degree 1 part isomorphic to $V$. Taking as generators of $B$ the elements $y_i=h x_i$, for some basis $x_1,\ldots,x_m$ of $V$, we can decompose the relations of $B$ as $G$-modules with respect to the generators $y_i$ in the same way as we did for the generators $x_1,\ldots,x_m$. But these 2 decompositions will be the same as $H'$ acts on the set of simple representations of $G$, but $A$ was also a $H'$-algebra. 
\end{proof}
One would expect that the centralizer of $G$ in $\wis{GL}(V)$ acts trivially on $\V_k$, but this will not be the case in general. However, if $V$ is a simple representation, then this is obviously true.
\par If $A=\C[V]$ or $A = \wedge V$, then $H=H'$ and we have an action of $H$ on $\V_k$ for each $k\geq 2$.
\begin{example}
Let $G = \T_2$ and $V = \chi_{e_1} \oplus \chi_{e_2}$. Then the $\T_2$-deformations of $\C[V]$ are parametrized by $\PP^1$ with $[a:b] \in \PP^1$ corresponding to the algebra
$$
T(V)/(ax_1x_2-bx_2x_1).
$$
The normalizer of $\T_2 \subset \wis{GL}(V)$ is the semidirect product $\T_2 \rtimes \Z_2$. $\T_2$ acts trivially on this moduli space, so the action boils down to the action of $\Z_2$ on $\PP^1$ defined by $s \cdot [a:b] = [b:a]$, which indeed give isomorphic algebras.
\label{ex:1dquantum}
\end{example}
\section{Some constructions}
\subsection{Twisting}
The main $G$-deformations up to degree $k$ we want to study are those of the polynomial ring $\C[V] = T(V)/(V \wedge V)$. Although the decomposition of $V \wedge V$ and $V \otimes V$ can be completely arbitrary, we can bound the dimension from below using twisting.
\begin{mydef}
Let $B$ be a graded algebra and $\beta \in \Aut(B)$ an automorphism that preserves the gradation. Then the twist $B^\beta$ is defined as the graded associative algebra with underlying set the elements of $B$ and multiplication rule
$$
\forall a \in B_k, b \in B_l: a *_\beta b = a\beta^k(b).
$$
\end{mydef}
\begin{proposition}
Let $\V_k$ be the variety parametrizing $G$-deformations up to degree $k$ of $\C[V]$ and let $V \cong \oplus_{i=1}^n S_i^{\oplus e_i}$ be the decomposition of $V$ in simple $G$-representations. 
We then have $$\dim \V_k \geq -1+\sum_{i=1}^n e_i^2.$$
In particular, $V_k$ contains a subvariety isomorphic to 
$$
\wis{PGL}_{e_1,\ldots,e_n}(\C) = \left(\prod_{i=1}^n \wis{GL}_{e_i}(\C)\right)/\C^*
$$
\end{proposition}
\begin{proof}
Let $\alpha \in \Aut_G(V) \cong \prod_{i=1}^n\wis{GL}_{e_i}(\C)$. Then the map
$$
\xymatrix{f_\alpha=Id \otimes \alpha^{-1}:V \otimes V \ar[r]& V\otimes V}
$$
is also a $G$-isomorphism, and $f_\alpha(V \wedge V)$ is the vector space of defining relations of $\C[V]^\alpha$. Therefore $f_\alpha(V \wedge V)\cong V \wedge V$. As $\C[V]$ has the property
$$
\forall \alpha,\beta \in \Aut(\C[V]) : f_\alpha(V \wedge V) = f_\beta(V \wedge V) \Rightarrow \exists \lambda \in \C^*: \alpha = \lambda \beta,
$$
we have that each $\alpha \in  \left(\prod_{i=1}^n\wis{GL}_{e_i}(\C)\right)/\C^*$ defines a $G$-deformation of $\C[V]$.
\end{proof}
It can be the case that $G$-deformations up to to degree 2 of $\C[V]$ are (generically) twists of $\C[V]$, as the next example will show.
\begin{example}
Let $G=\mathbb{T}_2=\C^* \times \C^*$ be the 2-dimensional torus and let $V = \chi_{e_1} \oplus \chi_{e_2}$. We then have the decomposition
\begin{gather*}
V \otimes V \cong \chi_{2e_1} \oplus \chi_{2e_2} \oplus \chi_{e_1 + e_2}^2,\\
V \wedge V \cong \chi_{e_1 + e_2}.
\end{gather*}
So the $G$-deformations up to degree 2 of $\C[V]$ are parametrized by $\PP^1$. The twists of $\C[V]$ that commute with the action of $\mathbb{T}_2$ is $\mathbb{T}_2$ itself, so the twists give a 1-dimensional family of $G$-deformations of $\C[V]$.
\par In fact, every point of $\PP^1$ defines a $G$-deformation of $\C[V]$, as for each point~$p=[a:b]\in \PP^1$ and $A^p = \C\langle x,y \rangle/(axy-byx)$, we have $H_{A^p}(t) = \frac{1}{(1-t)^2}$ and using Corollary \ref{cor:constchar}, we conclude that
$$
\forall p \in \PP^1 \forall k  \in \N: (A^p)_k \cong \C[V]_k \cong \oplus_{i=0}^k \chi_{(k-i,i)}. 
$$
\end{example}
If we are just interested whether we can make twists of $A$ that are also $G$-algebras but not necessarily $G$-deformations, then we can improve this theorem.
\begin{theorem}
Let $A=T(V)/(R)$ be a $G$-algebra and let $\phi \in \Aut_{\C^*}(A)$ such that
$$
\forall g \in G: [\phi,g] \in Z(\wis{GL}(V)).
$$
Then the twist $A^\phi$ is also a $G$-algebra with $A^\phi \cong V$.
\end{theorem}
\begin{proof}
Let $R^\phi$ be the relations of $A^\phi$ and let $\phi_k$ be the vector space automorphism on $V^{\otimes k}$ defined by 
$$
\phi_k(x_{i_1}\ldots x_{i_k})= x_{i_1} \phi(x_{i_2})\ldots\phi^{k-1}(x_{i_k})
$$
with $x_j \in V$ and extending this linearly. From the construction of $A^\phi$, it follows that 
$$
f \in R_k \Leftrightarrow \phi^{-k}(f) \in R^\phi_k
$$
We need to show that, for any $g \in G, g\cdot \phi^{-k}(f)\in R^\phi_k$. Let $[\phi,g]=\lambda \in \C^*$. We then have
$$
g \cdot \phi^{-k}(f) = \lambda^m \phi^{-k}(g\cdot f), m=\binom{k}{2}
$$
and as $g\cdot f \in R_k$, we have $g \cdot \phi^{-k}(f) \in R^\phi_k$.
\end{proof}
However, as the next example will show, it generally is not true that $A^\phi \cong A$ as graded $G$-module.
\begin{example}
Let $D_4 = H_2$ the Heisenberg group of order 8 (which is the same as the dihedral group of order 8) and take the Schr\"odinger representation of $D_4$ $V = \C x \oplus \C y$, defined by the matrices
$$
e_1 \mapsto \begin{bmatrix}
0 & 1 \\ 1 & 0
\end{bmatrix}, e_2 \mapsto \begin{bmatrix}
1 & 0 \\ 0 & -1
\end{bmatrix}.
$$
Then the induced map $\xymatrix{H_2 \ar[r]& \wis{PGL}_2(\C)}$ has as kernel the group generated by $[e_1,e_2]$ and $H_2/[e_1,e_2] \cong \Z_2 \times \Z_2$. So for example $e_2$ satisfies the condition of the previous theorem. Therefore, let $A = \C[V]$, then $A^{e_2} \cong \C\langle x,y \rangle/(xy+yx)$ ($A^{e_2}$ is the twist of $A$ by the automorphism $e_2$, not the subalgebra of $A$ fixed by $e_2$).
\par But $\C[x,y] \not\cong \C\{x,y\}$ as $H_2$-representations, as $\C[x,y] \cong \chi_{1,1}$ while $\C\{x,y\}\cong \chi_{0,1}$. So $A \not\cong A^{e_2}$.
\end{example}
\subsection{Ore extensions}
Another way to make $G$-deformations of an algebra $A$ is by using Ore extensions. Recall that $A[t;\sigma,\delta]$ with $\sigma \in \Aut(A)$ and $\delta \in \Der^{\sigma}(A)$.
\begin{proposition}
Let $A = T(V)/(R)$ be a quadratic $G$-algebra, $\sigma \in \Aut_G(A)$ and $\delta \in \Der^\sigma(A)$ such that
\begin{align}
\forall g \in G, \forall x \in A_1: \delta(g(x)) = \chi^*(g) g(\delta(x)).
\label{req:ore}
\end{align}
Then $A[t;\sigma,\delta]$ is a $G$-algebra such that $\deg(t) = 1$ and $\C t \cong \chi$. Conversely, if $A[t;\sigma,\delta]$ is a $G$-algebra such that $\C t\cong \chi$, then $\sigma \in \Aut_G(A)$ and $\delta \in \Der^\sigma(A)$ fulfils the above property.
\end{proposition}
\begin{proof}
The relations of an Ore extension of a quadratic algebra $A$ are of the form
$$
\forall x \in A_1: tx - \sigma(x)t - \delta(x)=0. 
$$
In order for $G$ to act on such an extension such that $\C t \cong \chi$, we need to calculate $g(tx - \sigma(x)t - \delta(x))$ for each $x \in A_1$. We have
\begin{align*}
g(tx - \sigma(x)t - \delta(x)) &= g(t)g(x) - g(\sigma(x))g(t) - g \delta(x) \\
                             &= \chi(g)t g(x) - \sigma(g(x)) \chi(g)t - \chi(g) \delta(g(x))\\
                             &=\chi(g)(tg(x) - \sigma(g(x))t- \delta(g(x))) = 0.
\end{align*}
For the other arrow, let $A[t;\sigma,\delta]$ be a $G$-algebra such that $\deg(t) = 1$ and $\C t \cong \chi$. Observe that the extra relations to get from $A$ to $A[t;\sigma,\delta]$ lie in the finite dimensional vector space $\C t \otimes V \oplus V \otimes \C t \oplus A_2$, which is also a direct sum as $G$-representations. These relations then form a $G$-subrepresentation if and only if the 2 maps
\begin{gather*}
\xymatrix{\C t \otimes V \ar[r]^-f& V \otimes \C t}\\
\xymatrix{\C t \otimes V \ar[r]^-g& A_2}
\end{gather*}
with $f$ mapping $t \otimes v$ to $\sigma(v) \otimes t$ and $g$ mapping $t \otimes v$  to $\delta(v) \in A_2$ are $G$-morphisms. So $\sigma \in \Aut_G(A)$ and $\delta$ fulfils the requirements of the proposition.
\end{proof}
Unfortunately, sometimes the only Ore extensions of $\C[V]$ we can make are those with $\delta = 0$.
\begin{proposition}
If $\sigma \in \prod_{i=1}^n\wis{GL}_{e_i}(\C)$ has no eigenvalue equal to 1 and if there is no $G$-submodule of $V$ isomorphic to $\chi$, then there exists no $0 \neq \delta \in \Der^\sigma(\C[V])$ fulfilling requirement \ref{req:ore}.
\label{prop:Ore}
\end{proposition}
\begin{proof}
Let $V = \oplus_{i=0}^{n-1}\C x_i$ as vector space. From the definition of a $\sigma$-derivation, we need
$$
\forall 0 \leq i,j \leq n-1: \sigma(x_i)\delta(x_j) + x_j \delta(x_i) = \delta(x_i x_j) = \delta(x_j x_i) = \sigma(x_j)\delta(x_i)+x_i \delta(x_j).
$$
From which it follows that
$$
(\sigma(x_i)-x_i)\delta(x_j) = (\sigma(x_j)-x_j)\delta(x_i).
$$
As $\sigma$ has no eigenvalue equal to 1, the elements $v_i = \sigma(x_i)-x_i=(\sigma -1)(x_i),i=1\ldots n$ also form a basis of $V$ and we have
$$
v_i \delta (x_j)= v_j \delta(x_i).
$$
Using the fact that $\C[V]$ is an UFD and that $v_1,\ldots,v_n$ form a basis, it follows that $\delta(x_i) = v v_i$ for some $v \in V$, so $\delta(x_i)=v(\sigma - 1)(x_i)$. By requirement \ref{req:ore}, we need that
$$
v (\sigma-1)(g(x)) =\delta(g(x)) =  \chi^*(g)g(\delta(x))=\chi^*(g) g(v)g((\sigma-1)(x)),
$$
from which it follows that
$$
g(v) = \chi(g) v.
$$
So if $\chi$ is not a subrepresentation of $V$, then necessarily $\delta = 0$.
\end{proof}

\section{$S_{n+1}$-deformations of $\C[V]$}
Let $V = S \oplus T$ be the $n+1$-dimensional permutation representation of $S_{n+1}$ with $S$ $n$-dimensional and $T$ the trivial representation. We have the following proposition.
\begin{proposition}
Let $n \geq 2$. The $S_{n+1}$-deformations up to degree 2 of $\C[V]$ are parametrized  by $\PP^2$.
\end{proposition}
\begin{proof}
We have $ V\otimes V = S \otimes S \oplus S^{\oplus 2} \oplus T$ and $V \wedge V = S \wedge S \oplus S$. $S \otimes S = S\wedge S \oplus T \oplus S \oplus W$ with $W$ simple, $W \neq S, S \wedge S$ according to \cite[Exercise 4.19]{fulton1991representation}. Then the $S_{n+1}$-deformations of $\C[V]$ are parametrized by $\Emb_{S_{n+1}}(S \wedge S \oplus S, S \wedge S \oplus S^{\oplus 3}) \cong \PP^2$. Let $V = \oplus_{i=0}^n \C x_i$ such that 
$$
\forall \sigma \in S_{n+1}: \sigma(x_i)= x_{\sigma(i)},
$$
put $y_i = x_0 - x_i, 1\leq i \leq n$ and let $v = \sum_{i=0}^n x_i$. Then for $P=[A:B:C] \in \PP^2$, the relations of a $S_{n+1}$-deformation $A_P$ of $\C[V]$ correspond to
\[
\begin{cases}
A y_i v + B v y_i + C y_i\left( (n-1)y_i-2\sum_{j=1,j\neq i}^n y_j\right)=0, 1 \leq i \leq n,\\
[y_i,y_j] = 0, 1 \leq i < j \leq n.
\end{cases}
\]
The theorem follows.
\end{proof}
From the relations, it generically follows that these algebras are Ore extensions of the commutative ring generated by $y_1,\ldots,y_n$. Unfortunately, if $\frac{B}{A}\neq -1$ and $C \neq 0$, the algebra $\C[y_1,\ldots,y_n]$ is not a commutative polynomial, as the corresponding automorphism of $\C[y_1,\ldots,y_n]$ does not have eigenvalue 1 as in Proposition \ref{prop:Ore}.
\par Therefore, the $S_{n+1}$-deformations of $\C[V]$ are parametrized by 2 lines, one corresponding to the differential polynomial rings $\C[y_1,\ldots,y_n][v,\delta_\alpha]$ with $\delta_a(y_i)=a y_i\left( (n-1)y_i-2\sum_{j=1,j\neq i}^n y_j\right)$ and one corresponding to a skew polynomial ring $\C[y_1,\ldots,y_n][x;\sigma_a]$ with $\sigma_a(y_i) = a y_i$.
\begin{theorem}
The Artin-Schelter regular algebras of global dimension $n+1$ which are $S_{n+1}$-deformations of the polynomial ring $\C[V]$ and are domains are parametrized by $\mathbf{V}(xy)\subset \A^1$, with one line corresponding to quantum algebras and the other line corresponding to differential polynomial rings.
\end{theorem}
\begin{proof}
Both lines define Ore extensions of the polynomial ring $\C[S]$, so we can apply the results of \cite{phan2012yoneda}. The points at infinity will be domains, as there will be zero divisors in degree 1. 
\end{proof}
The two strata we will study will correspond to the algebras with relations
\[
\begin{cases}
y_i v - a v y_i, 1 \leq i \leq n,\\
[y_i,y_j] = 0, 1 \leq i < j \leq n,
\end{cases}
\]
and
\[
\begin{cases}
vy_i-y_iv= c y_i\left( (n-1)y_i-2\sum_{j=1,j\neq i}^n y_j\right), 1 \leq i \leq n,\\
[y_i,y_j] = 0, 1 \leq i < j \leq n.
\end{cases}
\]

\subsection{Classification of the simple objects in $\wis{Proj}(\mA)$}
Let $\mA$ be a AS-regular $S_{n+1}$-deformation of $\C[V]$. In order to calculate the simple elements of $\wis{Proj}(\mA)$, we can use that in both cases, the sequence $y_1,y_2,\ldots,y_n$ is a normalizing, regular sequence.
\subsubsection{The Skew polynomial case}
Here we can use the results of \cite{2015point}.
\begin{theorem}
$\mA$ is a twist of the polynomial ring $\C[V]$. Consequently, $\wis{Proj}(\mA) \cong \PP^{n}$.
\end{theorem}
\begin{proof}
Take $\beta (y_i) = y_i, \beta v = av$ with $a \in \C^*$. Then the twist $\C[V]^\beta$ has relations
$$
y_i *_\beta y_j = y_i y _j = y_j *_\beta y_i, y_i *_\beta v = a y_i v = a v*_\beta y_i, 1 \leq i < j  \leq n.
$$
which corresponds to one of the lines in $\PP^2$ we defined.
\end{proof}
\subsubsection{The differential polynomial ring case}
By rescaling $v$, we may assume that we are working with the derivation defined by
$$
\delta(y_i) = y_i\left((n-1)y_i-2\sum_{j=1,j\neq i}^n y_j\right)=y_i f_i, i = 1 \ldots n.
$$
\begin{theorem}
The simple objects of $\wis{Proj}(\mA)$ are parametrized by $2^{n}-1$ lines through one point in $\PP^{n} = \PP((\mA_1)^*)$. Each point on such a line corresponds to a point module of $\mA$. If $n$ is even, the only point module with (non-trivial) finite dimensional simple quotients is the intersection point of these lines. If $n$ is odd, then there are $\binom{n}{\frac{n+1}{2}}$-lines parametrizing $\C^*$-orbits of 1-dimensional representations, which are the only lines with non-trivial simple quotients. The equations for these lines are determined by
$$
\V(y_i y_j (y_i-y_j)|1\leq i < j \leq n)
$$
There are no other simple objects in $\wis{Proj}(\mA)$.
\end{theorem}
\begin{proof}
Let $P \in \wis{Proj}(\mA)$ be a simple object, that is, $P$ is a graded module with Hilbert series $\frac{p}{(1-t)}$ for some $p \in \N, p \geq 1$ and each graded quotient of $P$ is finite dimensional. As each $y_i$ is normalizing, either $y_i \in \Ann(P)$ or $P$ corresponds to a simple $\mA[y_i^{-1}]_0$-representation, cfr. \cite{NastaFVO}. Assume that each $y_i\in \Ann(P)$, then $P$ is a $\mA/(y_1,\ldots, y_n) \cong \C[v]$-module and therefore $P \cong \C[v]$ as $\mA$-module.
\par Assume now that $y_1 \notin \Ann(P)$, then $P$ corresponds to a simple representation of $\mA[y_1^{-1}]_0$. Let $v_j = y_j y_1^{-1},2\leq j \leq n,w=vy_1^{-1}$. Then the relations of $\mA[y_1^{-1}]_0$ become $v_j v_k - v_k v_j=0$ and for $2\leq k \leq n$
\begin{align*}
w v_k &= v y_1^{-1} y_k y_1^{-1}\\
	  &= (y_1^{-1} v - \left((n-1)-2\sum_{j=2}^n y_jy_1^{-1}\right)) y_k y_1^{-1}\\
	  &=y_1^{-1}vy_ky_1^{-1}-(n-1)v_k+2\sum_{j=2}^n v_j v_k\\
	  &=y_1^{-1}(y_k v +y_k\left((n-1)y_k-2\sum_{j=1,j\neq k}^n y_j \right))y_1^{-1}-(n-1)v_k+2\sum_{j=2}^n v_j v_k\\
	  &=v_kw+(n-1)(v_k^2-v_k)-2\sum_{j=1,j\neq k}^n v_j v_k + 2\sum_{j=2}^n v_j v_k\\
	  &=v_kw + (n+1)v_k(v_k-1).
\end{align*}
Consequently, the 1-dimensional representations of $\mA[y_1^{-1}]_0$ correspond to the Zariski closed subset $\V(a_k(a_k-1)|2\leq k \leq n)\subset \A^n$, which is the union of $2^{n-1}$ lines.
\par We can now do the same for the algebra $(\mA[y_2^{-1}])_0$, which will give an additional $2^{n-1}$ lines. However, there are already $2^{n-1}-2^{n-2}$ lines found $\mA[y_1^{-1}]_0$ (those not annihilated by $y_1$), so we get $2^{n-2}$ new lines. By induction, we get in total
$$
\sum_{j=0}^{n-1} 2^{j} = 2^{n}-1
$$
lines. In $\PP^n$, each of these lines goes through the point $[0:0:\ldots:0:1]$, which will be the intersection point.
\par In order to prove that there are no fat point modules (simple objects with Hilbert series $\frac{p}{1-t}$, $p >1$), we will prove that $\mA[y_1^{-1}]_0$ has no $p$-dimensional representations for $p>1$. This will be enough, for $\mA[y_1^{-1}]_0 \cong \mA[y_i^{-1}]_0$ for any $1 \leq i \leq n$. We can rescale $v$ such that we are looking at the differential polynomial ring $\C[v_2,\ldots,v_n][w,\delta]$ with $\delta(v_i) = v_i(v_i-1)$.
\par First, we prove a lemma.
\begin{lemma}
Let $\rho$ be a simple representation of the algebra $\C\langle x,y \rangle/(ux-xu-x(x-1))$. Then $\rho$ is 1-dimensional and $\rho(x) = 0$ or $\rho(x)=1$.
\end{lemma}
\begin{proof}
If this is not so, then $\rho(x)-1$ and $\rho(x)$ are invertible. We then have for $Y = 1-\rho(x)^{-1}$ and $U = \rho(u)$
$$
YU - UY = \rho(u)\rho(x)^{-1}-\rho(x)^{-1}\rho(u) = -(1-\rho(x)^{-1})=-Y,
$$
so the couple $(U,Y)$ forms a representation of the 2-dimensional Heisenberg Lie algebra. If the lemma is not true, $Y$ is also invertible. Let $a$ be an eigenvector of $U$ with eigenvalue $\alpha$. We then have
$$
UYa = (YU+Y)a = (\alpha +1)Ya.
$$
But $\rho$ was finite dimensional, so $Y$ has a non-trivial kernel, which is a contradiction.
\par Let $\rho_0 = \{a \in \rho | \rho(x)a=0 \}$ and $\rho_1 = \{a \in \rho | \rho(x)a=a \}$ and take elements $a \in \rho_0$, $b \in \rho_1$. We then have
\begin{align*}
\rho(x)\rho(u)a = \rho(u)\rho(x)a - \rho(x)(\rho(x) - 1)a = 0 \\
\rho(x)\rho(u)b = \rho(u)\rho(x)b - \rho(x)(\rho(x) - 1)b = b 
\end{align*}
So $\rho = \rho_0$ and $\rho_1=0$ or vice versa as $\rho$ is simple, in both cases $\rho(x(x-1))=0$ and $\rho$ is actually a simple representation of $\C[x,u]/(x(x-1))$. The lemma follows.
\end{proof}
As the $v_j, 2 \leq j \leq n$ commute with each other, let $a$ be an eigenvector of each $v_j$ with eigenvalue either 1 or 0 for each $j$, so $v_j a = \alpha_j a$ with $\alpha_j = 0$ or $\alpha_j = 1$ for some simple representation $\rho$ of $\mA[y_1^{-1}]_0$. We then have
$$
v_j w a = w v_j a + (v_j)(v_j-1)a = v_j w a.
$$
As $\mA[y_1^{-1}]_0 a = \rho$, it follows that the images of $w$ and $w_j$ commute, which means that $\rho$ is only simple if $\rho$ is 1-dimensional.
\par Now we need to check which of these point modules have simple quotients. This amounts to check which of these point modules have under the shift functor $\xymatrix{\wis{Proj}(\mA) \ar[r]^-{[1]}&\wis{Proj}(\mA)}$ a finite orbit. Equivalently, as $\mA$ is Artin-Schelter regular, we need to check which orbits of the associated automorphism $\phi$ on the point variety are finite, cfr. \cite{smith19924}. It is clear that under $\phi$ the intersection point is sent to itself, as it is the only singular point of the point variety. Let $P$ be a point module, not the intersection point, then $P$ corresponds to a set $T \subset \{1,\ldots,n\}, T \neq \{1,\ldots,n\}$ such that
$ y_i \in \Ann(P) \Leftrightarrow i \in T$. In addition, we have $y_i,y_j \notin \Ann(P) \Rightarrow y_i-y_j \in \Ann(P)$. Let $[a_1:\ldots:a_n:r]$ be the corresponding point of $\PP^n$, so $a_i = 0 \Leftrightarrow i \in T$ and $i,j \not\in T \Rightarrow a_i = a_j$ and $r$ arbitrary. We need to find $[b_1:\ldots:b_n,s]$ such that the following equations hold
$$
\begin{cases}
rb_i-a_is = a_i\left((n-1)b_i-2\sum_{j=1,j\neq i}^n b_j \right), 1\leq i \leq n,\\
a_ib_j-a_jb_i=0, 1\leq i < j \leq n.
\end{cases}
$$
From the second set of equations, it follows that $b_i = a_i \forall 1\leq i \leq n$. We may pick $a_i= b_i = 1$ if $i \not\in T$. Consequently, the first set of equations are fulfilled if $i \in T$. If $i \not\in T$, then we have
$$
s = r-\left((n-1)-2\sum_{j=1,j\neq i}^n a_j \right)=r-n-1+2|T|
$$
If $|T| \neq \frac{n+1}{2}$, this shows that $\phi$ is a translation on $\PP^1_{[a_i:r]}$, which fixes one point (the intersection point of all the lines). However, if $|T| = \frac{n+1}{2}$ (so $n$ has to be odd), then $\phi$ fixes each point on the corresponding line, so this will give a line parametrizing $\C^*$-orbits of 1-dimensional representations of $\mA$.
\end{proof}
For $n=2,3$, we can make things more explicit.
\begin{example}
As $S_3 \cong D_3 = \langle e_1,e_2 | e_1^3=e_2^2=1, e_2e_1e_2=e_1^2\rangle$, the differential polynomial ring $\mA$ we are studying is isomorphic to $\C\langle x,y,t\rangle/(R)$ with relations
$$
\begin{cases}
xy-yx=0,\\
xt-tx=y^2,\\
yt-ty=x^2.
\end{cases}
$$ 
In this case, the point modules can be computed by the method of multilinearization as the zeroset of the polynomial
$$
\det\left(\begin{bmatrix}
-y_0 & x_0 & 0 \\ -t_0 & -y_0 & x_0 \\ -x_0& -t_0 & y_0 
\end{bmatrix}\right)=y_0^3-x_0^3.
$$
This is indeed the union of three lines through the point $[0:0:1]$. If we take for example the line $x_0 = y_0$, then the automorphism defined by this algebra on $\V(x_0-y_0)$ is given by $[x_0:x_0:t_0]\mapsto [x_0:x_0:t_0+x_0]$, which is clearly an automorphism of infinite order and fixes only one point.
\par The fact that all 1-dimensional simple representations come from the intersection point follows immediately as we have $\mA/([x,t],[y,t])\cong \C[a,b,c]/(a^2,b^2)$.
\end{example}
\begin{example}
For $n=3$, we have the isomorphism $S_4 \cong (\Z_2 \times \Z_2) \rtimes S_3$. Let $e_1,e_2$ be the natural generators of $\Z_2 \times \Z_2$. If $V=S \oplus T$ is the permutation representation of $S_4$, then we can find a basis $\{v_{i,j}|0\leq i,j \leq 1\}$ of $V$ such that
$$
e_1 \cdot v_{i,j} = (-1)^i v_{i,j}, e_2 \cdot v_{i,j} = (-1)^j v_{i,j}.
$$
As $S_4$-module, the decomposition is given by $\C v_{0,0} \oplus S_4 \cdot v_{1,0}$. Using this basis, the differential polynomial ring has relations (under a suitable isomorphism)
$$
\begin{cases}
[v_{1,0},v_{0,1}]=[v_{1,1},v_{1,0}]=[v_{0,1},v_{1,1}]=0,\\
v_{1,0} v_{0,0} - v_{0,0} v_{1,0}= v_{0,1}v_{1,1},\\
v_{0,1} v_{0,0} - v_{0,0} v_{0,1}= v_{1,0}v_{1,1},\\
v_{1,1} v_{0,0} - v_{0,0} v_{1,1}= v_{1,0}v_{0,1}.
\end{cases}
$$
Again using multilinearization, we find that the point modules are parametrized by 7 lines, given by the Zariski closed subset in $\PP^3$ defined by
$$
\V\left(V_{1,1}(V_{1,0}^2-V_{0,1}^2),V_{1,0}(V_{0,1}^2-V_{1,1}^2),V_{0,1}(V_{1,0}^2-V_{1,1}^2)\right).
$$
The point modules with 1-dimensional simple quotients can again be easily found, by taking the quotient $\mA/([\mA,\mA])\cong \C[a,b,c,d]/(bc,bd,cd)$. This ring is clearly the coordinate ring of 3 affine planes, intersecting 2 by 2.
\end{example}

\end{document}